\documentclass[12pt,twoside]{amsart}
\usepackage{amsfonts}
\usepackage{amsmath}
\usepackage{amssymb}
\usepackage{graphicx}
\usepackage{mathrsfs}
\usepackage{multicol}
\usepackage{color}
\usepackage{pst-eucl,pstricks-add}
\usepackage{pst-3dplot}% en 3D
\usepackage{pst-solides3d}

\usepackage{lscape}

\newtheorem{thm}{Theorem}
\newtheorem{cor}{Corollary}
\newtheorem{lem}{Lemma}
\newtheorem{prop}{Proposition}
\theoremstyle{definition}

\theoremstyle{remark}
\newtheorem{rem}{\bf Remark\/}

\numberwithin{equation}{section}

\definecolor{bronze}{rgb}{0.4, 0.7, 0.1}

\def\1{{\mathchoice {\rm 1\mskip-4mu l} {\rm 1\mskip-4mu l}{\rm 1\mskip-4.5mu l} {\rm 1\mskip-5mu l}}}
\newcommand{\ds}{\displaystyle}
\usepackage[width=16.00cm, height=21.00cm, left=2.00cm, right=2.00cm, top=4cm, bottom=2cm]{geometry}

\title{Toeplitz operators via Carleson measures On $\beta$-modified Bergman Spaces}
	\author{Safa Snoun}
	\email{snoun.safa@fsg.rnu.tn}
	\address{University of Gabes\\ Faculty of Sciences of Gabes\\ LR17ES11 Mathematics and Applications laboratery\\ 6072, Gabes, Tunisia.}
	
\begin{document}
    \begin{abstract}
      In this paper, we study Bergman projection $\mathbb{P}_{\alpha,\beta}$ and Toeplitz operators  $T^{\alpha,\beta}_\varphi$ on the $\beta$-modified Bergman space $\mathcal{A}_{\alpha,\beta}^p$. We give some properties of $\mathbb{P}_{\alpha,\beta}$ and a necessary and sufficient condition for $T^{\alpha,\beta}_\varphi$ to be compact. We end with a characterization of Toeplitz operators via Carleson measures by introducing a new Bergman metric inherited by the Bergman kernel $\mathbb{K}_{\alpha,\beta}$ that will be equivalent to the classical Bergman-Poincar\'e metric.
	\end{abstract}
		%%%%%%%%%%%%%%%%%%%%%%%%%%%%%%%%%%%%%%%%%%%%%%%%%%%%%%%%%%%%%%%%%%%%%%%%%%%%%%%%%
		
		\maketitle

\section{Introduction}
		
	Berezin Transform, Toeplitz and Hankel operators on Bergman spaces are typical examples of operators that received some attention during the last period. Whereas the theory of these operators on the Hardy space is by now well understood and  becomes classical and have long been explored (cf. e.g.\cite{r6} and the bibliography given therein). However, only some few years ago, researchers began investigating these operators on the Bergman space. Nowadays, there are rich theories that combine between operator theory and Bergman spaces, specially on the behavior of Toeplitz, Hankel operators and Berezin transform. Well-known partial results are included in  \cite{r2,r1,r7,r8,St,r10,r11,r12,r9}.\\

In the current paper, we are concerned about the behavior of these operators on a new type of Bergman space, the so-called $\beta$-modified Bergman space introduced by N.Ghiloufi and M.Zaway in \cite{Gh-Za}. Let $\mathbb{D}$ be the open unit disk in the complex plane $\mathbb{C}$. For $-1<\alpha,\beta<+\infty$, the $\beta$-modified Bergman space $\mathcal{A}_{\alpha,\beta}^p=\mathcal{A}_{\alpha,\beta}^p(\mathbb{D}^*)$ consists of those holomorphic functions $f$ on the unit disk $\mathbb{D}^*$ such that
$$\left\|f\right\|_{\alpha,\beta,p}=\bigg(\int_{\mathbb{D}}|f(z)|^p\ d\mu_{\alpha,\beta}(z)\bigg)^{1/p}<+\infty,$$
where \begin{equation}\label{eq0}
\displaystyle{d\mu_{\alpha,\beta}(z)=\frac{1}{\mathscr{B}(\alpha+1,\beta+1)}|z|^{2\beta} (1-|z|^2)^\alpha dA(z)},
\end{equation}
 $\mathscr{B}$ is the beta function and $dA$ is the normalized area measure on $\mathbb{D}$. The space $\mathcal{A}_{\alpha,\beta}^2$ is a closed subspace of the Hilbert space $\textbf{L}^2(\mathbb{D},d\mu_{\alpha,\beta})$ with inner product given by \begin{equation}\label{eq1}
\left\langle f,g\right\rangle_{\alpha,\beta}=\int_{\mathbb{D}} f(z)\ \overline{g(z)}\ d\mu_{\alpha,\beta}(z),
\end{equation}
for every $f, g \in \textbf{L}^2(\mathbb{D},d\mu_{\alpha,\beta})$ and with reproducing (Bergman) kernel given by
\begin{equation}\label{eq2}
 \mathbb{K}_{\alpha,\beta}(w,z)=\frac{Q_{\alpha,\beta}(\xi)}{\xi^m(1-\xi)^{2+\alpha}}:=\mathcal{K}_{\alpha,\beta}(w\overline{z})
\end{equation}
 where 				
			$$\begin{array}{lcl}
				\mathcal{K}_{\alpha,\beta}(\xi)&=&\ds \frac{\mathscr B(\alpha+1,\beta+1)}{\mathscr B(\alpha+1,\beta_0+1)} \xi^{-m}\ {_2F_1}\left( \begin{array}{c} 1 , \alpha + \beta_0  + 2 \\ \beta_0+1 \end{array}\bigg | \xi \right)\\
				&=&\ds \frac{\mathscr B(\alpha+1,\beta+1)}{\mathscr B(\alpha+1,\beta_0+1)}\frac{1}{\xi^{m}(1-\xi)^{\alpha+2}}\ {_2F_1}\left( \begin{array}{c} \beta_0 , -(\alpha+1) \\ \beta_0+1 \end{array}\bigg | \xi \right)
				\end{array}$$
		 Here ${_2F_1}$ is the hypergeometric function and $\beta_0=\beta-m\in]-1,0]$ such that $m$ is a non-negative integer. Thus we deduce that  	
		
        $$Q_{\alpha,\beta}(\xi)=\left\{
            \begin{array}{lcl}
            (\alpha+1)\mathscr B(\alpha+1,\beta+1)& if &\beta\in\mathbb N\\
            \ds\frac{\mathscr B(\alpha+1,\beta+1)}{\mathscr B(\alpha+1,\beta_0+1)}\sum_{n=0}^{+\infty}\frac{\beta_0(-\xi)^n}{n+\beta_0}{\alpha+1\choose n} &if&\beta\not\in\mathbb N
          \end{array}\right.
        $$

The subspace $\mathcal{A}_{\alpha,\beta}^2$ is a Hilbert space and $\mathcal{A}_{\alpha,\beta}^2=\mathcal{A}_{\alpha,m}^2$ if $\beta=\beta_0+m$ with $m\in \mathbb{N}$ and $-1<\beta_0\leq 0$.  We claim here that $\mathcal{A}_{\alpha}^2(\mathbb{D})=\mathcal{A}_{\alpha,\beta_0}^2$ is the classical Bergman space equipped with the new norm $\left\|.\right\|_{\alpha,\beta_0,2}$. For more details about the properties of this Kernel and this space, one can refer to \cite{Gn-Sn,Gh-Za,Gh-Sn}. \\

Sometimes, we need at computation level to estimate the function $|Q_{\alpha,\beta}|$ by a constant. To this aim we consider $\mathcal{J}_{\alpha}$  the set of $\beta\in ]-1,0]$ such that $Q_{\alpha,\beta}$ has no zero in $\mathbb{D}$. This set is not empty because in \cite[Theorem 3.1]{Gh-Sn}, it has been shown that there exists $\beta_\alpha\in ]-1,0[$ such that $Q_{\alpha,\beta}$ has no zero in $\mathbb{D}$ for every $\beta_0\in]\beta_\alpha,0]$.\\

In the characterization of the compact Toeplitz operators on Bergman spaces, the M$\ddot{o}$bius transformations on $\mathbb{D}$ will play a crucial role: for $z\in \mathbb{D}$,  the automorphism $\psi_z$ is defined by
\begin{equation}\label{eq4}
\psi_z(w)=\frac{z-w}{1-\overline{z}w},
\end{equation}
which verifies the following properties:
\begin{enumerate}
	\item $\psi_z^{-1}=\psi_z$.
	\item The real Jacobian determinant of $\psi_z$ at $w$ is $\ds |\psi'_z(w)|^2=\frac{(1-|z|^2)^2}{|1-\overline{z}w|^4}$.
	\item $\ds 1-|\psi_z(w)|^2=\frac{(1-|z|^2)(1-|w|^2)}{|1-\overline{z}w|^2}$.
\end{enumerate}
One of the main issue I encountered in this context is that the measures under consideration  are not invariant by the set of disk automorphisms. This makes the situation more challenging and some properties do not work here.\\

The remaining sections are organized as follows. Section 2 deals with the necessary and sufficient conditions for the boundedness of projection on $\textbf{L}^p(\mathbb D,d\mu_{\alpha,\beta})$ and some properties of projection on the considered space is also presented in this section. Section 3 is devoted to the study of the Toeplitz operators by giving a conditions to be compact operator. We characterize  in the last section  the continuity and the compactness of Toeplitz operators via Carleson measures.

\section{Continuity of projections on $\textbf{L}^p(\mathbb D,d\mu_{\alpha,\beta})$}
For $-1<\alpha,\beta<+\infty$, let $\mathbb{P}_{\alpha,\beta}$ be the orthogonal projection from $\textbf{L}^2(\mathbb D,d\mu_{\alpha,\beta})$ onto $\mathcal A_{\alpha,\beta}^2$. Then for every $f\in \textbf{L}^2(\mathbb D,d\mu_{\alpha,\beta})$ we have
    $$\begin{array}{lcl}
         \mathbb{P}_{\alpha,\beta}f(z)&=&\ds \langle f,\mathbb{K}_{\alpha,\beta}(.,z)\rangle_{\alpha,\beta}=\int_{\mathbb D}f(w)\frac{Q_{\alpha,\beta}(z\overline{w})}{(z\overline{w})^m(1-z\overline{w})^{2+\alpha}}d\mu_{\alpha,\beta}(w)\\
         &=&\ds \frac{1}{z^m} \mathbb{P}_{\alpha,\beta_0}[w^mf(w)](z).
      \end{array}
    $$
    with $\beta=\beta_0+m$ where $m\in\mathbb N$ and $-1<\beta_0\leq0$. That is $\mathbb{P}_{\alpha,\beta}=M^{-1}\circ \mathbb{P}_{\alpha,\beta_0}\circ M$ where $M$ is the linear operator:
    $$\begin{array}{lccl}
      M:&\mathcal A_{\alpha,\beta}^2&\longrightarrow& \mathcal A_{\alpha,\beta_0}^2\\
      & f& \longmapsto & \ds \frac{\mathscr B(\alpha+1,\beta_0+1)}{\mathscr B(\alpha+1,\beta+1)}z^mf.
  \end{array}$$
 The last integral formula of $ \mathbb{P}_{\alpha,\beta}$ suggested that we can apply $ \mathbb{P}_{\alpha,\beta}$ to a function in $\textbf{L}^p(\mathbb{D},d\mu_{\alpha,\beta})$  whenever $1\leq p< +\infty$.

In this part, we will give necessary and sufficient  conditions for the  boundedness of $\mathbb P_{\alpha,\beta}$ from $\textbf{L}^p(\mathbb{D},d\mu_{a,b})$ onto $\mathcal{A}_{a,b}^p$.\\
In the rest of the paper $\theta\lesssim \eta$ near a point means that there exists $c>0$ such that $\theta(z)\leq c\eta(z)$ in a neighborhood of this point and $\theta\approx \eta$ means that  $\theta\lesssim \eta$ and  $\eta\lesssim \theta$.
The following lemma will be a fundamental tool for proofs.
\begin{lem}\label{l3}
For every $-1<\sigma,\gamma,\alpha,\beta<+\infty$ with $\beta=\beta_0+m$, $m\in \mathbb{N}$ and $\beta_0\in\mathcal J_\alpha$, we set
$$I_\omega(z):=\int_{\mathbb{D}}\frac{|Q_{\alpha,\beta}(z\overline{w})|(1-|w|^2)^\sigma|w|^{2\gamma}}{|1-z\overline{w}|^{2+\sigma+\omega}} dA(w).$$
Then, we have
$$I_\omega(z)\approx \left\{
\begin{array}{lcl}
1 & if&\omega<0\\
\log(\frac{1}{1-|z|^2})& if& \omega=0\\
\frac{1}{(1-|z|^2)^\omega}&if& \omega>0
\end{array}
\right.$$
as $|z|\rightarrow 1^-$.
\end{lem}
\begin{proof}
     Using the same techniques used in \cite{He-KO-Zh} to prove Theorem $7$ and the hypothesis on $Q_{\alpha,\beta}$ cited in Section $5$ of this paper, one can prove the result.
\end{proof}

On account of \cite[Theorem 3]{Gh-Za}, the projection $\mathbb{P}_{\alpha,m}$ is  bounded  from  $L^p(\mathbb{D},d\mu_{a,b})$ onto $\mathcal{A}_{a,b}^p(\mathbb{D}^*)$  if and only if $p(\alpha+1)>(a+1)$ and $mp-2<2b<mp-2+2p$, when $p>1$ and if $m-2<2b\leq m$, when $p=1$. To prove the general case, we need the following theorem:
\begin{thm}\label{thm1}
Let $-1<a,b,\alpha,\beta<+\infty$ with $\beta=\beta_0+m$, $m\in \mathbb{N}$ and $\beta_0\in \mathcal J_\alpha$. We define the two integral operators $A$ and $B$ by
$$Af(z)=\frac{1}{z^m}\int_{\mathbb{D}} \frac{f(w)\ Q_{\alpha,\beta}(z\overline{w})\ (1-|w|^2)^{\alpha-a}\ w^m}{|w|^{2m+2b-2\beta} (1-z\overline{w})^{2+\alpha}}\ d\mu_{a,b}(w)$$
and
$$Bf(z)=\frac{1}{|z|^m}\int_{\mathbb{D}} \frac{f(w)\ |Q_{\alpha,\beta}(z\overline{w})|\ (1-|w|^2)^{\alpha-a}\ |w|^{2\beta-2b-m} }{\ |1-z\overline{w}|^{2+\alpha}}\ d\mu_{a,b}(w).$$
Then for $1\leq p<+\infty$, the following assertions are equivalent:
\begin{enumerate}
\item  $A$ is bounded on $\textbf{L}^p(\mathbb{D},d\mu_{a,b})$
\item  $B$ is bounded on $\textbf{L}^p(\mathbb{D},d\mu_{a,b})$
\item  $p(\alpha+1)>(a+1)$ and
$$\left\{
\begin{array}{lcl}
m-2<2b\leq 2\beta-m& if& p=1\\
pm-2<2b<p(2\beta+2)-pm-2& if& p>1.
\end{array}
\right.$$
\end{enumerate}
\end{thm}

\begin{proof}
 It's clear that the boundedness of $B$ on $\textbf{L}^p(\mathbb{D},d\mu_{a,b})$ implies the boundedness of $A$.  As for proving  that $B$ is bounded on $\textbf{L}^p(\mathbb{D},d\mu_{a,b})$ when $A$ is also bounded, it suffices to use the following transformation:
$$\Phi_{z} f(w)=\frac{(1-z\overline{w})^{2+\alpha}\ |Q_{\alpha,\beta}(z\overline{w})| |w|^m}{|1-z\overline{w}|^{2+\alpha}\ Q_{\alpha,\beta}(z\overline{w}) w^m} f(w).$$
Now, we assume that $B$ is bounded on $\textbf{L}^p(\mathbb{D},d\mu_{a,b})$ and we apply this operator to a function of the form $f_N (z)=(1-|z|^2)^N$, where $N$ is sufficiently large. Then, we obtain
$$\ds\left\|B f_N\right\|^p_{a,b,p}=\ds\int_{\mathbb{D}} \frac{(1-|z|^2)^a\ |z|^{2b-pm}}{\mathscr{B}^{p+1}(a+1,b+1)}\left( \int_{\mathbb{D}} \frac{(1-|w|^2)^{N+\alpha} |Q_{\alpha,\beta}(z\overline{w})||w|^{2\beta-m}}{|1-z\overline{w}|^{2+\alpha}}dA(w)\right)^p dA(z)$$
is finite. So, accordingly to Lemma \ref{l3}, we conclude that
\begin{equation}\label{eq1}
 2b> pm-2.
\end{equation}

To prove the others inequalities, we suppose first that $p>1$ and let $q$ be its conjugate exponent. Let $B^\star$ be the adjoint operator of $B$ with respect to the inner product $\langle.,.\rangle_{a,b}$ which is given by
\begin{align*}
B^\star g(z)&= (1-|z|^2)^{\alpha-a}|z|^{2\beta-m-2b} \int_{\mathbb{D}} \frac{g(w)\ |Q_{\alpha,\beta}(z\overline{w})|}{|w|^{m} |1-z\overline{w}|^{2+\alpha}}\ d\mu_{a,b}(w)\\
&=\frac{(1-|z|^2)^{\alpha-a}|z|^{2\beta-m-2b}}{\mathscr{B}(a+1,b+1)} \int_{\mathbb{D}} \frac{g(w)\ |Q_{\alpha,\beta}(z\overline{w})|\ |w|^{2b-m} (1-|w|^2)^a}{ |1-z\overline{w}|^{2+\alpha}}dA(w).
\end{align*}
We apply  $B^\star$ to the function $f_N$ defined above, we obtain that
$$\begin{array}{lcl}
     \left\|B^\star f_N\right\|^q_{a,b,q}&=&\ds M\int_{\mathbb{D}} (1-|z|^2)^{a+q(\alpha-a)} |z|^{2b+q(2\beta-m-2b)}\times\\
     &&\ds\hfill\left( \int_{\mathbb{D}} \frac{(1-|w|^2)^{N+a} |Q_{\alpha,\beta}(z\overline{w})|\ |w|^{2b-m}}{|1-z\overline{w}|^{2+\alpha}} dA(w)\right)^qdA(z)
  \end{array}
$$
is finite, with $\ds M=\frac{1}{\mathscr{B}^{q+1}(a+1,b+1)}$. Whence, by using again Lemma \ref{l3} together with Equation (\ref{eq1}), we get
$$pm-2<2b<p(2\beta+2)-pm-2\quad and\quad p(\alpha+1)>a+1.$$
Now, if we suppose that $p=1$, then $B^\star$ is bounded on $\textbf{L}^\infty(\mathbb{D})$. If we  act $B^\star$ on the constant function $g\equiv 1$, we obtain
$$ \sup_{z\in \mathbb{D}^*}\ \frac{(1-|z|^2)^{\alpha-a}|z|^{2\beta-m-2b}}{\mathscr{B}(a+1,b+1)} \int_{\mathbb{D}} \frac{|Q_{\alpha,\beta}(z\overline{w})|\ |w|^{2b-m}\ (1-|w|^2)^{a}}{ |1-z\overline{w}|^{2+\alpha}}\ dA(w)<+\infty.$$
Then, by applying Lemma \ref{l3},  we find what we want to show: $\alpha-a>0$, $2\beta-m-2b\geq0$ and $2b-m>-2$.\\

Conversely, for $p=1$ the result follows from Lemma \ref{l3}. Now, if $p>1$ and $pm-2<2b<p(2\beta+2)-pm-2$, let we prove that  $B$ is bounded on $\textbf{L}^p(\mathbb{D},d\mu_{a,b})$. We put $$h(z)=\frac{1}{|z|^t(1-|z|^2)^s},\ t, s\in \mathbb{R}$$ and
$$\psi(z,w)=\frac{ (1-|w|^2)^{\alpha-a}\ |Q_{\alpha,\beta}(z\overline{w})| |w|^{2\beta-m-2b}}{|z|^{m} |1-z\overline{w}|^{2+\alpha}}.$$
By virtue of Lemma \ref{l3}, if we assume that
$$ 0<s<\frac{\alpha+1}{q},\qquad \frac{m}{q}\leq t<\frac{2\beta-m+2}{q}$$
then
$$\begin{array}{l}
\ds\int_{\mathbb{D}}  h(w)^q\ \psi(z,w)\ d\mu_{a,b}(w)=\\
=
\ds\frac{1}{\mathscr{B}(a+1,b+1)}\int_{\mathbb{D}} \frac{ (1-|w|^2)^{\alpha-qs}\ |Q_{\alpha,\beta}(z\overline{w})|\ |w|^{2\beta-m-qt}}{|z|^{m} |1-z\overline{w}|^{2+\alpha}} dA(w)\\
\ds\leq \frac{C_1}{ |z|^{m} (1-|z|^2)^{sq}}=C_1.|z|^{tq-m} h(z)^q.
\end{array}$$
If we assume that
$$ \frac{a-\alpha}{p}<s<\frac{a+1}{p},\qquad \frac{2b+m-2\beta}{p}\leq t<\frac{2b-m+2}{p}$$
then
$$\begin{array}{l}
\ds\int_{\mathbb{D}}  h(z)^p\ \psi(z,w) d\mu_{a,b}(z)=\\
=\ds\frac{ (1-|w|^2)^{\alpha-a} |w|^{2\beta-m-2b}}{\mathscr{B}(a+1,b+1)}\int_{\mathbb{D}} \frac{(1-|z|^2)^{a-ps}|Q_{\alpha,\beta}(z\overline{w})||z|^{2b-m-pt}}{|1-z\overline{w}|^{2+\alpha}} dA(z)\\
\ds\leq \frac{C_2.\ |w|^{2\beta-m-2b}}{  (1-|w|^2)^{sp}}=C_2.|w|^{2\beta-m-2b+tp} h(w)^p.
\end{array}$$
Hypothesis $(3)$ gives that
$$\left]\frac{m}{q},\frac{2\beta-m+2}{q}\right[\cap\left]\frac{2b+m-2\beta}{p},\frac{2b-m+2}{p}\right[\neq\emptyset,$$
$$\left]0,\frac{\alpha+1}{q} \right[\cap\left]\frac{a-\alpha}{p},\frac{a+1}{p}\right[\neq \emptyset$$
which shows the existence of $t$ and $s$ satisfying the inequalities above. Thus, an application of Shur's test implies that $B$ is  bounded on  $\textbf{L}^p(\mathbb{D},d\mu_{a,b})$.
\end{proof}
\begin{cor}\label{cor1}
Suppose $-1<a,b,\alpha,\beta<+\infty$ with $\beta=\beta_0+m$, $m\in \mathbb{N}$ and $\beta_0\in \mathcal J_\alpha$. Let $1\leq p<+\infty$, then $\mathbb{P}_{\alpha,\beta}$ is a bounded projection from $\textbf{L}^p(\mathbb{D},d\mu_{a,b})$ onto $\mathcal{A}_{a,b}^p$ if and only if $p(\alpha+1)>(a+1)$ and $$\left\{
\begin{array}{lcl}
m-2<2b\leq 2\beta-m& if& p=1\\
pm-2<2b<p(2\beta+2)-pm-2& if& p>1.
\end{array}
\right.$$
\end{cor}
\begin{proof}
It's an immediate consequence of Theorem \ref{thm1}.
\end{proof}

 We claim here that if $\alpha=a$ and $\beta=b$, then  $\mathbb{P}_{\alpha,\beta}$ is a bounded projection from $\textbf{L}^p(\mathbb{D},d\mu_{\alpha,\beta})$ onto $\mathcal{A}_{\alpha,\beta}^p$ if and only if
$$\frac{2(\beta+1)}{2(\beta+1)-m}<p<\frac{2(\beta+1)}{m}.$$
Of course if $m=0$ (means $\beta=\beta_0\in \mathcal J_\alpha$), these inequalities are reduced to $p>1$.

Now, we give some general properties for the orthogonal projection $\mathbb{P}_{\alpha,\beta}$ on the $\beta$-modified Bergman space $\mathcal A_{\alpha,\beta}^2$.

\begin{lem}\label{l1}
Let $-1<\alpha,\beta<+\infty$ with $\beta=\beta_0+m$ and $m\in \mathbb{N}$ and $\mathbb{P}_{\alpha,\beta}$ be the orthogonal projection on $\mathcal A_{\alpha,\beta}^2$. Then, for $s,t$ be nonnegative integers, we have
$$\mathbb{P}_{\alpha,\beta}(\overline{z}^s z^{t-m})=\left\{
\begin{array}{lcl}
\ds \frac{(\alpha+\beta_0+t+2)_{-s}}{(\beta_0+t+1)_{-s}} z^{t-s-m} & if& t\geq s\\
0& if& t< s
\end{array}
\right.$$
where $(a)_n=a(a+1)\dots(a+n-1)=\ds\frac{\Gamma(a+n)}{\Gamma(a)}$ is the Pochhammer symbol.
\end{lem}
\begin{proof}
Let $z\in \mathbb{D}^*$, then
$$\begin{array}{lcl}
\mathbb{P}_{\alpha,\beta}(\overline{z}^s z^{t-m})&=&\ds \frac{z^{-m}}{\mathscr{B}(\alpha+1,\beta_0+1)}\int_{\mathbb{D}} \overline{w}^{s-m}w^{t-m}  _2F_1\left(\left.
\begin{array}{c}
1,\alpha+\beta_0+2\\
\beta_0+1
\end{array}\right|\overline{w}z\right) |w|^{2\beta}(1-|w|^2)^{\alpha} dA(w)\\
&=&\ds \frac{z^{-m}}{\mathscr{B}(\alpha+1,\beta_0+1)}\sum_{n=0}^{+\infty}\frac{(\alpha+\beta_0+2)_n}{(
\beta_0+1)_n} z^n \int_{\mathbb{D}} \overline{w}^{s-m+n} w^{t-m}  |w|^{2\beta}(1-|w|^2)^{\alpha} dA(w).\\
\end{array}$$
In order to evaluate the last integral, we use polar coordinates
$$\begin{array}{lcl}
\int_{\mathbb{D}} \overline{w}^{s-m+n} w^{t-m}  |w|^{2\beta}(1-|w|^2)^{\alpha} dA(w)&=&\ds \frac{1}{\pi}\int_{0}^1 r^{s+t-2m+2\beta+n+1} (1-r^2)^\alpha dr \int_{0}^{2\pi} e^{i(t-s-n)\theta} d\theta\\
&=& \left\{
\begin{array}{lcl}
\ds \mathscr{B}(\alpha+1, t+\beta_0+1) & if& n=t-s\\
0& if& n\neq t-s
\end{array}
\right.$$
\end{array}$$
Therefore and by a simple computation we find the desired result.
\end{proof}
\begin{lem}\label{l2}
Let $-1<\alpha,\beta<+\infty$ with $\beta=\beta_0+m$ and $m\in \mathbb{N}$ and $\mathbb{P}_{\alpha,\beta}$ be the orthogonal projection on $\mathcal A_{\alpha,\beta}^2$. Then, for every non-negative integer $s$, we have
$$\left\|\mathbb{P}_{\alpha,\beta}\left(\overline{z}^s \sum_{k=0}^{+\infty} a_k z^{k-m}\right)\right\|^2_{\alpha,\beta}=
\ds \sum_{k=0}^{+\infty} \frac{(\beta+1)_{t-s-m}}{(\alpha+\beta+2)_{t-s-m}}\frac{(\alpha+\beta_0+k+2)^2_{-s}}{(\beta_0+k+1)^2_{-s}} |a_k|^{2}
$$
\end{lem}
\begin{proof}
In account of Lemma \ref{l1}, it is easy to see that $\mathbb{P}_{\alpha,\beta}\left(\overline{z}^s \sum_{k=0}^{s-1} a_k z^{k-m}\right)= 0$ thus it suffices to compute the norm for every $k\geq s $.  A classical computation gives:
$$\begin{array}{lcl}
\ds\left\|\mathbb{P}_{\alpha,\beta}\left(\overline{z}^s \sum_{k=s}^{+\infty} a_k z^{k-m}\right)\right\|^2_{\alpha,\beta}&=&\ds \frac{1}{\mathscr{B}(\alpha+1,\beta+1)}\sum_{k=s}^{+\infty}\sum_{j=s}^{+\infty} \frac{(\alpha+\beta_0+k+2)_{-s}}{(
\beta_0+k+1)_{-s}}\frac{(\alpha+\beta_0+j+2)_{-s}}{(
\beta_0+j+1)_{-s}} a_k \overline{a}_j\\
&&\ds \times \int_{\mathbb{D}} z^{k-s-m}\overline{z}^{j-s-m} |z|^{2\beta}(1-|z|^2)^{\alpha} dA(z).\\

\end{array}$$
Making the change to polar coordinates, we found
$$\begin{array}{lcl}
\ds \int_{\mathbb{D}} \overline{z}^{j-s-m}z^{k-s-m} |z|^{2\beta}(1-|z|^2)^{\alpha} dA(z)&=& \left\{
\begin{array}{lcl}
\ds \mathscr{B}(\alpha+1, k-s+\beta_0+1) & if& j= k\\
0& if& j\neq k\\
\end{array}
\right.
\end{array}$$
and the proof of Lemma \ref{l2} is accomplished.
\end{proof}
\section{Toeplitz operators on $\mathcal{A}_{\alpha,\beta}^p$}

Let $-1<\alpha,\beta<+\infty$ with $\beta=\beta_0+m$ and $m\in \mathbb{N}$. For $\varphi\in L^\infty(\mathbb{D})$, we define the Toeplitz operator $T^{\alpha,\beta}_\varphi$ on $\mathcal{A}_{\alpha,\beta}^p$ by $T^{\alpha,\beta}_\varphi(f)=\mathbb{P}_{\alpha,\beta}(\varphi f)$, where $\mathbb{P}_{\alpha,\beta}$ is the projection from $\textbf{L}^p(\mathbb{D},d\mu_{\alpha,\beta})$ onto $\mathcal{A}_{\alpha,\beta}^p$. In view of Corollary \ref{cor1} in section 2 and for $\beta_0\in\mathcal J_\alpha$, we have that $T^{\alpha,\beta}_\varphi$ is bounded on $\mathcal{A}_{\alpha,\beta}^p$ if and only if  $$\frac{2(\beta+1)}{2(\beta+1)-m}<p<\frac{2(\beta+1)}{m}.$$
Again if $m=0$ ($\beta=\beta_0\in\mathcal J_\alpha$) these inequalities are reduced to $p>1$.
In the sequel, if there is no ambiguity, we use $T_\varphi$ instead of $T^{\alpha,\beta}_\varphi$ and we fix $p$ such that $T_\varphi$  is continuous on $\mathcal{A}_{\alpha,\beta}^p$, i.e. the previous inequalities are satisfied.

\begin{prop}
Let $-1<\alpha,\beta<+\infty$ with $\beta =\beta_0+m$, $m\in \mathbb{N}$ and let $\varphi, \varphi_1, \varphi_2\in L^\infty(\mathbb{D},d\mu_{\alpha,\beta})$. Then

\begin{enumerate}
	\item $T_{\varphi_1+\varphi_2}=T_{\varphi_1}+T_{\varphi_2}.$
	\item $T_{\overline{\varphi_1}}T_{\varphi_2}=T_{\overline{\varphi_1}\varphi_2}$ if $\varphi_1$ or $\varphi_2$ is analytic.
	\item $(T_\varphi)^*=T_{\overline{\varphi}}$ where $T_{\overline{\varphi}}$ is the Toeplitz operator on $\mathcal{A}_{\alpha,\beta}^q$, with $q$ is the conjugate exponent of $p$.
\end{enumerate}

\end{prop}
\begin{proof}
The first statement is obvious by using the linearity of integral. To prove the second one, we suppose that $\varphi_1$ is analytic and we fix
 $f\in \mathcal{A}_{\alpha,\beta}^p$ and $g\in \mathcal{A}_{\alpha,\beta}^q$. Since $\mathbb{P}_{\alpha,\beta}(g)=g$ and using the fact that $\varphi_1$ is analytic and by Fubini's Theorem, we perform
\begin{align*}
\left\langle T_{\overline{\varphi_1}}T_{\varphi_2} f,g\right\rangle_{\alpha,\beta}&=\int_{\mathbb{D}} \bigg[\int_{\mathbb{D}}\overline{\varphi_1}(w)T_{\varphi_2} f(w)\overline{\mathcal{K}_{\alpha,\beta}(w\overline{z})}d\mu_{\alpha,\beta}(w)\bigg]\ \overline{g}(z)d\mu_{\alpha,\beta}(z).\\
&= \int_{\mathbb{D}} \overline{\varphi_1}(w)T_{\varphi_2} f(w)\overline{g(w)}d\mu_{\alpha,\beta}(w)\\
&=\int_{\mathbb{D}}\overline{\varphi_1}(w) \bigg[\int_{\mathbb{D}}\overline{\varphi_2}(\xi)f(\xi)\overline{\mathcal{K}_{\alpha,\beta}(\xi\overline{w})}d\mu_{\alpha,\beta}(\xi)\bigg]\ \overline{g}(w)d\mu_{\alpha,\beta}(w)\\
&=\int_{\mathbb{D}}\varphi_2(\xi)f(\xi)\overline{\varphi_1}(\xi) \overline{g}(\xi)d\mu_{\alpha,\beta}(\xi)\\
&= \left\langle \overline{\varphi_1}\varphi_2 f,\mathbb{P}(g)\right\rangle_{\alpha,\beta}\\
&=\left\langle T_{\overline{\varphi_1}\varphi_2}f, g\right\rangle_{\alpha,\beta}.
\end{align*}

Now, for the last one we fix
 $f\in \mathcal{A}_{\alpha,\beta}^p$ and $g\in \mathcal{A}_{\alpha,\beta}^q$. Then,
\begin{align*}
\left\langle T_{\varphi} f,g\right\rangle_{\alpha,\beta}&=\int_{\mathbb{D}} T_{\varphi} f(z) \overline{g(z)}\ d\mu_{\alpha,\beta}(z)\\
&=\int_{\mathbb{D}} \bigg[\int_{\mathbb{D}}\frac{\varphi(w)f(w)Q_{\alpha,\beta}(z\overline{w})}{(z\overline{w})^m(1-z\overline{w})^{\alpha+2}}d\mu_{\alpha,\beta}(w)\bigg]\ \overline{g}(z)d\mu_{\alpha,\beta}(z).\\
\end{align*}
Applying Fubini's Theorem and using the fact that $\mathbb{P}_{\alpha,\beta}(g)=g$ and $\mathbb{P}_{\alpha,\beta}(f)=f$, since  $f\in \mathcal{A}_{\alpha,\beta}^p$ and $g\in \mathcal{A}_{\alpha,\beta}^q$, and the fact that $P_{\alpha,\beta}$ is self-adjoint with respect to the inner product associated with $d\mu_{\alpha,\beta}$,  we obtain
\begin{align*}
\left\langle T_{\varphi} f,g\right\rangle_{\alpha,\beta}
&=\int_{\mathbb{D}}\varphi(w)f(w) \bigg[\int_{\mathbb{D}}\frac{\overline{g}(z) Q_{\alpha,\beta}(z\overline{w})}{(z\overline{w})^m(1-z\overline{w})^{\alpha+2}}d\mu_{\alpha,\beta}(z)\bigg]\ d\mu_{\alpha,\beta}(w)\\
&=\int_{\mathbb{D}}\varphi(w)f(w) \overline{g}(w) d\mu_{\alpha,\beta}(w)\\
&=\int_{\mathbb{D}}\varphi(w)f(w) \overline{g}(w) d\mu_{\alpha,\beta}(w)\\
&= \left\langle f,\overline{\varphi}g\right\rangle_{\alpha,\beta}\\
&= \left\langle f,\mathbb{P}(\overline{\varphi}g)\right\rangle_{\alpha,\beta}\\
&=\left\langle f,T_{\overline{\varphi}}g\right\rangle_{\alpha,\beta}.
\end{align*}
Therefore $(T_\varphi)^*=T_{\overline{\varphi}}.$
\end{proof}

\begin{lem}\label{lem4.2}
Let $-1<\alpha,\beta<+\infty$ with $\beta =\beta_0+m$, $m\in \mathbb{N}$ and $\beta_0\in \mathcal J_\alpha$. Assume that $\varphi\in L^\infty(\mathbb{D})$ be a function with compact support. Then $T_\varphi$ is a compact operator on $\mathcal{A}_{\alpha,\beta}^p$.
\end{lem}
\begin{proof}
Suppose that  $Supp\ \varphi=K$ is a compact subset of $\mathbb{D}$ and let $(f_n)_{n\geq0}$ be a bounded sequence in $\mathcal{A}_{\alpha,\beta}^p$. Let $m_{p,\beta}$ be the integer defined in \cite{Gh-Za} by
$$m_{p,\beta}=\left\{\begin{array}{lcl}
                        \ds\left\lfloor\frac{2(\beta+1)}{p}\right\rfloor & if & \ds\frac{2(\beta+1)}{p}\not\in\mathbb N\\
                        \ds\frac{2(\beta+1)}{p}-1& if & \ds\frac{2(\beta+1)}{p}\in\mathbb N
                      \end{array}\right.$$
 such that if $f\in\mathcal A_{\alpha,\beta}^p$ then 0 is a pole of $f$ with order $\nu_f=\nu_f(0)$ that satisfies
$\nu_f\leq m_{p,\beta}$. Then for every $f\in \mathcal{A}_{\alpha,\beta}^p$ the function  $\widetilde{f}=z^{\nu_f} f$ is holomorphic on $\mathbb D$.  Using Proposition $1$ in \cite{Gh-Za}, we obtain that the sequence $(\widetilde{f}_n)_n$ is uniformly bounded on each compact subset of $\mathbb{D}$. On account of Montel's theorem, there exists a subsequence $(\widetilde{f}_{n_k})_k$ of  $(\widetilde{f}_n)_{n\geq0}$ that converges uniformly  to $\widetilde{f}$ on $K$. Therefore if we set $$C=\frac{\mathscr{B}(\alpha+1,\beta-\frac{p\nu_f}{2}+1)}{\mathscr{B}(\alpha+1,\beta+1)}$$ we obtain
\begin{align*}
\left\|\varphi f_{n_k}-\varphi f\right\|^p_{\alpha,\beta,p}&=C \left\|\varphi \tilde{f}_{n_k}-\varphi \tilde{f}\right\|^p_{\alpha,\beta-\frac{p\nu_f}{2},p}\\
&=C \int_{K} |\varphi(z)|^p\ |\tilde{f}_{n_k}(z)- \tilde{f}(z)|^p d\mu_{\alpha,\beta-\frac{p\nu_f}{2}}(z)\\
&\leq C \left\|\varphi\right\|_{\infty}^p \sup_{z\in K}|\tilde{f}_{n_k}(z)- \tilde{f}(z)|^p \mu_{\alpha,\beta-\frac{p\nu_f}{2}}(K),
\end{align*}
Since $$\sup_{z\in K}|\tilde{f}_{n_k}(z)- \tilde{f}(z)|^p\rightarrow 0, \qquad k\rightarrow +\infty,$$ the sequence $(\varphi f_{n_k})_k$ converges in $\textbf{L}^p(\mathbb{D},d\mu_{\alpha,\beta})$ to $\varphi f$. Now since the projection $\mathbb P_{\alpha,\beta}$ is continuous on $\textbf{L}^p(\mathbb{D},d\mu_{\alpha,\beta})$ then  $\mathbb P_{\alpha,\beta}(\varphi f_{n_k})$ converges in $\mathcal{A}_{\alpha,\beta}^p$ to $\mathbb P_{\alpha,\beta}(\varphi f)$. Hence the operator $T_\varphi$ is compact on $\mathcal{A}_{\alpha,\beta}^p$.
\end{proof}

Applying Lemma \ref{l3}, we obtain

\begin{align*}
\left\| K_{\alpha,\beta}(.,z)\right\|_{\alpha,\beta,p}
&=\frac{1}{\mathscr{B}^{\frac{1}{p}}(\alpha+1,\beta+1)\ |z|^{m}}\bigg(\int_{\mathbb{D}}\frac{|Q_{\alpha,\beta}(z\overline{w})|^p\ (1-|w|^2)^\alpha\ |w|^{2\beta-mp}}{|1-z\overline{w}|^{(2+\alpha)p}} dA(w)\bigg)^{\frac{1}{p}}\\
&\approx \frac{1}{\mathscr{B}^{\frac{1}{p}}(\alpha+1,\beta+1)}\frac{1}{ |z|^{m}\ (1-|z|^2)^{(2+\alpha)(\frac{1}{q})}},\quad as\ |z|\rightarrow 1^-,
\end{align*}
where $q$ is the conjugate exponent of $p$.\\
Set $k_{\alpha,\beta,p}^z$ be the normalized reproducing kernel at $z\in \mathbb{D}^*$:
$$k_{\alpha,\beta,p}^z(w)=\frac{K_{\alpha,\beta}(w,z)}{\left\| K_{\alpha,\beta}(.,z)\right\|_{\alpha,\beta,p}}$$ which has the following property:
\begin{prop}\label{prop4.2}
$k_{\alpha,\beta,p}^z$ converges to $0$ weakly in  $\mathcal{A}_{\alpha,\beta}^p$ as $|z|$ tends to $1^-$.
\end{prop}
\begin{proof}

Let $g_h(z)=\frac{h(z)}{z^{m_{q,\beta}}}$ be the holomorphic function on $\mathbb{D}^*$ where $h$ is a holomorphic bounded function on $\mathbb{D}$. When $|z|$ tends to $1^-$, we obtain
\begin{equation}\label{eq41}
\langle g_h,k_{\alpha,\beta,p}^z\rangle_{\alpha,\beta}\approx \mathscr{B}^{\frac{1}{p}}(\alpha+1,\beta+1) |z|^{m-m_{q,\beta}}\ (1-|z|^2)^{(2+\alpha)(\frac{1}{q})} h(z),
\end{equation}
which is clearly tends to $0$. Since the set of functions $g_h$, with $h$ bounded, is dense in the Bergman space $\mathcal{A}_{\alpha,\beta}^q$, we conclude that $k_{\alpha,\beta,p}^z$ converges to $0$ weakly in  $\mathcal{A}_{\alpha,\beta}^p$ as $|z|\longrightarrow 1^-$.
\end{proof}

 This property will be useful in the following theorem when we announce a second result on the compactness of Toeplitz operators.
\begin{thm}\label{theo2}
Assume $-1<\alpha,\beta<+\infty$ with $\beta =\beta_0+m$, $m\in \mathbb{N}$ and $\beta_0\in \mathcal J_\alpha$. Let $\varphi\in C(\mathbb{\overline{D}})$. Then $T_\varphi$ is a compact operator on $\mathcal{A}_{\alpha,\beta}^p$ if and only if $\varphi|_{\partial \mathbb{D}}=0$.
\end{thm}

\begin{proof}
For the sufficient condition, we assume that $\varphi_{|\partial \mathbb{D}}\equiv0$. Then $\varphi$ can be uniformly approximated by functions with compact supports in $\mathbb{D}$. So, we deduce immediately from Lemma \ref{lem4.2} that  $T_\varphi$ is compact.

Now for the necessary condition, we suppose that $T_\varphi$ is compact on $\mathcal{A}_{\alpha,\beta}^p$.
We take $\xi\in \partial \mathbb{D}$. Thanks to  Proposition \ref{prop4.2}, we have $k_{\alpha,\beta,p}^z$ converges to $0$ weakly in  $\mathcal{A}_{\alpha,\beta}^p$ as $z$ tends to $\xi$. Therefore,
$$\langle T_\varphi k_{\alpha,\beta,p}^z,k_{\alpha,\beta,q}^z\rangle_{\alpha,\beta}=\left\langle T_\varphi \frac{\mathbb K_{\alpha,\beta}(.,z)}{\left\| \mathbb K_{\alpha,\beta}(.,z)\right\|_{\alpha,\beta,p}},\frac{\mathbb K_{\alpha,\beta}(.,z)}{\left\| K_{\alpha,\beta}(.,z)\right\|_{\alpha,\beta,q}} \right\rangle_{\alpha,\beta}\underset{z\to \xi}\longrightarrow 0.$$
Otherwise, one has
\begin{align*}
&\langle T_\varphi k_{\alpha,\beta,p}^z,k_{\alpha,\beta,q}^z \rangle_{\alpha,\beta}\\
&=\int_{\mathbb{D}}\bigg(\int_{\mathbb{D}}\frac{\varphi(\zeta)\ \mathbb K_{\alpha,\beta}(\zeta,z)\ Q_{\alpha,\beta}(w\overline{\zeta})}{(w\overline{\zeta})^m (1-w\overline{\zeta})^{\alpha+2}\ \left\| \mathbb K_{\alpha,\beta}(.,z)\right\|_{\alpha,\beta,p}}d\mu_{\alpha,\beta}(\zeta)\bigg) \frac{\overline{\mathbb K_{\alpha,\beta}(w,z)}}{\left\| \mathbb K_{\alpha,\beta}(.,z)\right\|_{\alpha,\beta,q}}d\mu_{\alpha,\beta}(w)\\
&=\int_{\mathbb{D}}\frac{\varphi(\zeta)\ \mathbb K_{\alpha,\beta}(\zeta,z)}{\left\| \mathbb K_{\alpha,\beta}(.,z)\right\|_{\alpha,\beta,p}\ \left\| \mathbb K_{\alpha,\beta}(.,z)\right\|_{\alpha,\beta,q}}\bigg(\int_{\mathbb{D}}\frac{\overline{\mathbb K_{\alpha,\beta}(w,z)}\ Q_{\alpha,\beta}(w\overline{\zeta})}{(w\overline{\zeta})^m (1-w\overline{\zeta})^{\alpha+2}\ }d\mu_{\alpha,\beta}(w)\bigg) d\mu_{\alpha,\beta}(\zeta)\\
&=\int_{\mathbb{D}}\frac{\varphi(\zeta)\ \mathbb K_{\alpha,\beta}(\zeta,z)}{\left\| \mathbb K_{\alpha,\beta}(.,z)\right\|_{\alpha,\beta,p}\ \left\| \mathbb K_{\alpha,\beta}(.,z)\right\|_{\alpha,\beta,q}}\ \overline{\mathbb K_{\alpha,\beta}(\zeta,z)} d\mu_{\alpha,\beta}(\zeta)\\
&\simeq  |z|^{2m}\ \int_{\mathbb{D}} \varphi(\zeta)\ |\mathbb K_{\alpha,\beta}(\zeta,z)|^2\ (1-|z|^2)^{(2+\alpha)}\ d\mu_{\alpha,\beta}(\zeta).\\
\end{align*}

Now, if we make the change of variable $\zeta=\varphi_z(w)$ in the last integral, we obtain

\begin{align*}
&\langle T_\varphi k_{\alpha,\beta,p}^z,k_{\alpha,\beta,q}^z \rangle_{\alpha,\beta}\simeq  \ \int_{\mathbb{D}} \varphi(\zeta)\ \frac{|Q_{\alpha,\beta}(\overline{z}\zeta)|^2(1-|z|^2)^{\alpha+2}}{|\zeta|^{2m} |1-\overline{z}\zeta|^{2(2+\alpha)}}d\mu_{\alpha,\beta}(\zeta)\\
&\simeq  \ \int_{\mathbb{D}} (\varphi\circ\varphi_z)(w)\ \frac{|Q_{\alpha,\beta}(\overline{z}\varphi_z(w))|^2\ |\varphi_z(w)|^{2\beta-2m}}{|w|^{2\beta}}  d\mu_{\alpha,\beta}(w).
\end{align*}
Since $\varphi_z(w)\longrightarrow \xi$, as $z\longrightarrow \xi$ then by applying the dominated convergence theorem, we find
$$\langle T^{\alpha,\beta}_\varphi k_{\alpha,\beta,p}^z,k_{\alpha,\beta,q}^z\rangle_{\alpha,\beta}\underset{z\to \xi}\simeq \varphi(\xi)|Q_{\alpha,\beta}(1)|^2 \int_{\mathbb{D}} \frac{1}{|w|^{2\beta}}  d\mu_{\alpha,\beta}(w).$$
Thus, $\varphi(\xi)=0$,  for every $\xi\in \partial\mathbb{D}$.
\end{proof}
	\section{Carleson measures on $\mathcal{A}_{\alpha,\beta}^p$}
	To more understand the properties of Toeplitz and Berezin operators and to see the relationship between them, we will define them in a more general situation. Indeed they will  be defined for what we call Carleson measures. In order to characterize the boundedness and compactness of these operators, we need to introduce a new Bergman metric inherited by the Bergman kernel $\mathbb{K}_{\alpha,\beta}$ and we compare it with the classical Bergman-Poincar\'{e} metric. Namely, we will prove that they are equivalent while the new one has a negative (non-constant) curvature.
\subsection{Modified Bergman-Poincar\'{e} metric}
The aim of this part is the construction of a new Bergman metric $\textbf{d}_{\alpha,\beta}$, called modified Bergman-Poincar\'{e} metric. Indeed, we will prove that the function $\kappa_{\alpha,\beta}(z):=\log(\mathbb K_{\alpha,\beta}(z,z))=\log(\mathcal K_{\alpha,\beta}(|z|^2))$ is well defined and $\mathcal C^\infty$ subharmonic on $\mathbb D$. Thus if we set
$$\varrho^2_{\alpha,\beta}(z)=\frac{\partial^2 \kappa_{\alpha,\beta}}{\partial z\partial\overline{z}}(z),$$
we obtain $\varrho^2_{\alpha,\beta}(z)\geq 0$ for every $ z\in\mathbb D.$ Before we assert the results of this subsection, we need a some preparation. The results in the following lemma are similar to the ones in \cite[Corollary 2.3]{Gh-Sn}, the only difference between them is the fact that the next result will be proved for all $-1<\beta<+\infty$ while the other is proved for $-1<\beta<0$. The proofs are almost the sames.
\begin{lem}\label{l4.1}
        For every $-1<\alpha,\beta<+\infty$, we set  $$H_{\alpha,\beta}(\xi)=\beta\sum_{n=0}^{+\infty}\frac{(-1)^n}{n+\beta}{\alpha+1\choose n}\xi^n.$$ Then $H_{\alpha,\beta}$ satisfies:
        $$\ds\xi H'_{\alpha,\beta}(\xi)=\ds \beta\left((1-\xi)^{\alpha+1}-H_{\alpha,\beta}(\xi)\right) $$
        and $$\begin{array}{lcl}
             H_{\alpha+1,\beta}(\xi)&=&\ds\frac{1}{\alpha+\beta+2}\left(\xi(1-\xi)H_{\alpha,\beta}'(\xi)+(\alpha+\beta+2-\beta\xi) H_{\alpha,\beta}(\xi)\right)\\
         &=&\ds\frac{1}{\alpha+\beta+2}\left((\alpha+2)H_{\alpha,\beta}(\xi)+\beta(1-\xi)^{\alpha+2}\right).
           \end{array}$$

    \end{lem}
    \begin{proof}
    For the first equality we have
    $$\begin{array}{lcl}
         \xi H'_{\alpha,\beta}(\xi)&=&\ds\beta\sum_{n=0}^{+\infty}\frac{n}{n+\beta}{\alpha+1\choose n}(-\xi)^n\\
         &=&\ds\beta\left(\sum_{n=0}^{+\infty}{\alpha+1\choose n}(-\xi)^n-\sum_{n=0}^{+\infty}\frac{\beta}{n+\beta}{\alpha+1\choose n}(-\xi)^n\right)\\
         &=&\ds \beta\left((1-\xi)^{\alpha+1}-H_{\alpha,\beta}(\xi)\right).
      \end{array}
    $$
    However, to prove the second equality we apply a famous equality as follows:
    $$\begin{array}{ll}
        &\ds\xi(1-\xi)H_{\alpha,\beta}'(\xi)+(\alpha+\beta+2-\beta\xi) H_{\alpha,\beta}(\xi)\\
         =&\ds \beta\sum_{n=0}^{+\infty}\frac{n}{n+\beta}{\alpha+1\choose n}(-\xi)^n+\beta\sum_{n=0}^{+\infty}\frac{n}{n+\beta}{\alpha+1\choose n}(-\xi)^{n+1}\\
         &\ds+(\alpha+\beta+2)\beta\sum_{n=0}^{+\infty}\frac{(-\xi)^n}{n+\beta}{\alpha+1\choose n}+\beta^2\sum_{n=0}^{+\infty}\frac{(-\xi)^{n+1}}{n+\beta}{\alpha+1\choose n}\\
         =&\ds \beta\left(\sum_{n=0}^{+\infty}\frac{n+\alpha+\beta+2}{n+\beta}{\alpha+1\choose n}(-\xi)^n+\sum_{n=0}^{+\infty}{\alpha+1\choose n}(-\xi)^{n+1}\right)\\
         =&\ds \beta\left(\sum_{n=0}^{+\infty}\frac{n+\alpha+\beta+2}{n+\beta}{\alpha+1\choose n}(-\xi)^n+\sum_{n=1}^{+\infty}{\alpha+1\choose n-1}(-\xi)^n\right)\\
         =&\ds (\alpha+\beta+2)+\beta\sum_{n=1}^{+\infty}\left(\frac{n+\alpha+\beta+2}{n+\beta}{\alpha+1\choose n}+{\alpha+1\choose n-1}\right)(-\xi)^n\\
         =&\ds (\alpha+\beta+2)\beta\sum_{n=0}^{+\infty}\frac{(-\xi)^n}{n+\beta}{\alpha+2\choose n}\\
         =&\ds(\alpha+\beta+2)H_{\alpha+1,\beta}(\xi)
           \end{array}$$
        The last equality can be deduced from the two previous ones.
    \end{proof}
		As a consequences of the above lemma, we have the following remarks:
    \begin{rem}\label{r1}
        For every $\alpha\in\mathbb N,\ -1<\beta \leq 0$ and $t\in[0,1[$ we have
        \begin{enumerate}
          \item $H_{\alpha,\beta}(t)\geq 1$
          \item $0\leq H_{\alpha,\beta+1}(t)\leq H_{\alpha,\beta}(t)$
        \end{enumerate}
    \end{rem}
    \begin{proof}
       The proof of the first statement is simple. For the second one, we will reason by induction on $\alpha$. We set $g_{\alpha,\beta}(t)=H_{\alpha,\beta+1}(t)- H_{\alpha,\beta}(t)$. Then, we have for $\alpha=0$
          $$\begin{array}{lcl}
               g_{0,\beta}(t)&=&\ds H_{0,\beta+1}(t)- H_{0,\beta}(t)=\left(1-\frac{\beta+1}{2+\beta}t\right)-\left(1-\frac{\beta}{1+\beta}t\right)\\
               &=&\ds\frac{-t}{(1+\beta)(2+\beta)}\leq 0.
            \end{array}
          $$
          Thus we obtain $$\frac{1}{\beta+2}\leq 1-\frac{\beta+1}{2+\beta}t=H_{0,\beta+1}(t)\leq H_{0,\beta}(t).$$
          We assume that $0\leq H_{\alpha,\beta+1}(t)\leq H_{\alpha,\beta}(t)$. Then thanks to Lemma \ref{l4.1},
          $$\begin{array}{lcl}
             H_{\alpha+1,\beta+1}(t)&=&\ds\frac{1}{\alpha+\beta+3}\left((\alpha+2)H_{\alpha,\beta+1}(t)+(\beta+1)(1-t)^{\alpha+2}\right)\\
             &\geq&\ds \frac{\alpha+2}{\alpha+\beta+3}H_{\alpha,\beta+1}(t)\geq H_{\alpha,\beta+1}(t)\geq0.
           \end{array}$$
           Moreover,
          $$\begin{array}{lcl}
             g_{\alpha+1,\beta}(t)&=&\ds H_{\alpha+1,\beta+1}(t)- H_{\alpha+1,\beta}(t)\\
             &=&\ds\frac{1}{\alpha+\beta+3}\left((\alpha+2)H_{\alpha,\beta+1}(t)+(\beta+1)(1-t)^{\alpha+2}\right)\\ &&-\ds\frac{1}{\alpha+\beta+2}\left((\alpha+2)H_{\alpha,\beta}(t)+\beta(1-t)^{\alpha+2}\right)\\
             &=&\ds \frac{(\alpha+2)g_{\alpha,\beta}(t)}{\alpha+\beta+2}-\frac{(\alpha+2)\left[H_{\alpha,\beta}(t)-(1-t)^{\alpha+2}\right]} {(\alpha+\beta+2)(\alpha+\beta+3)}\leq0.
           \end{array}$$
        %\end{enumerate}
    \end{proof}
		We assert now the main result:
\begin{thm}
    For every $-1<\alpha<+\infty$ and $-1<\beta\leq 0$ we have
    $$\varrho^2_{\alpha,\beta}(z)=-\beta\frac{(1-t)^\alpha}{H_{\alpha,\beta}^2(t)}\left((\alpha+\beta+1)H_{\alpha,\beta}(t)-\frac{\beta(\alpha+\beta+2)}{1+\beta} H_{\alpha,\beta+1}(t)\right)+\frac{\alpha+2}{(1-t)^2}$$
    where $t=|z|^2$.
    In particular, if $\alpha\in\mathbb N$ we obtain
    $$\frac{\alpha+2}{(1-|z|^2)^2}\leq \varrho^2_{\alpha,\beta}(z)\leq \frac{\alpha+\beta+2}{(1+\beta)}\frac{1}{(1-|z|^2)^2}.$$
\end{thm}

\begin{proof}
We claim that if $\varphi:[0,1[\longrightarrow ]0,+\infty[$ is a $\mathcal C^2-$function on $[0,1[$ and $\phi(z):=\log(\varphi(|z|^2))$ then for every $z\in\mathbb D$ we have
    \begin{equation}\label{a1}
    \frac{\partial^2\phi}{\partial z\partial \overline{z}}(z)=:\mathscr G_\varphi(|z|^2)=\left(\frac{\varphi'(|z|^2)}{\varphi(|z|^2)}+|z|^2\frac{\varphi''(|z|^2)}{\varphi(|z|^2)}-|z|^2\left(\frac{\varphi'(|z|^2)}{\varphi(|z|^2)} \right)^2\right).
    \end{equation}
    The function $\mathbb K_{\alpha,\beta}(z,z)>0$ for every $z\in\mathbb D$. Hence the function
    $$\kappa_{\alpha,\beta}(z)=\log(\mathbb K_{\alpha,\beta}(z,z))=\log(Q_{\alpha,\beta}(|z|^2))-(\alpha+2)\log(1-|z|^2)$$ is well defined and $\mathcal C^\infty$ on $\mathbb D$.
    It follows that for every $z\in\mathbb D$ we have
    $$\frac{\partial^2 \kappa_{\alpha,\beta}}{\partial z\partial\overline{z}}(z)=\frac{\partial^2 \log(Q_{\alpha,\beta}(|z|^2))}{\partial z\partial\overline{z}}+\frac{(\alpha+2)}{(1-|z|^2)^2}=\mathscr G_{Q_{\alpha,\beta}}(|z|^2)+\frac{(\alpha+2)}{(1-|z|^2)^2}.$$
    The term $\frac{(\alpha+2)}{(1-|z|^2)^2}$ is exactly the one corresponding to the classical Bergman-Poincar\'e metric. For the other term, using the formula \eqref{a1} and lemma \ref{l4.1}, it is not hard to see that
    $$\mathscr G_{Q_{\alpha,\beta}}(t)=-\beta\frac{(1-t)^\alpha}{H^2_{\alpha,\beta}(t)}\left((\alpha+\beta+1)H_{\alpha,\beta}(t)-\frac{\beta(\alpha+\beta+2)}{1+\beta} H_{\alpha,\beta+1}(t)\right)$$
    Thanks to Remark \ref{r1},
    $$\ds(\alpha+\beta+1)H_{\alpha,\beta}(t)-\frac{\beta(\alpha+\beta+2)}{1+\beta} H_{\alpha,\beta+1}(t)\geq \ds \frac{\alpha+1}{1+\beta}H_{\alpha,\beta+1}(t)\geq0.$$
    and
    $$\ds(\alpha+\beta+1)H_{\alpha,\beta}(t)-\frac{\beta(\alpha+\beta+2)}{1+\beta} H_{\alpha,\beta+1}(t)\leq \ds \frac{\alpha+1}{1+\beta}H_{\alpha,\beta}(t).$$
    Thus we obtain
    $$\mathscr G_{Q_{\alpha,\beta}}(t) \leq -\beta(\alpha+1)\frac{(1-t)^\alpha}{(1+\beta)H_{\alpha,\beta}(t)} \leq -\frac{\beta(\alpha+1)}{1+\beta}.$$
    It follows that
    $$\frac{\alpha+2}{(1-t)^2}\leq \varrho^2_{\alpha,\beta}(z)\leq -\frac{\beta(\alpha+1)}{1+\beta}+\frac{\alpha+2}{(1-t)^2}\leq \frac{\alpha+\beta+2}{(1+\beta)}\frac{1}{(1-t)^2}.$$
\end{proof}
As a simple consequence, for every $z\in\mathbb D$, we have $$\lim_{\beta\to0^-}\varrho_{\alpha,\beta}(z)=\frac{\sqrt{\alpha+2}}{1-|z|^2}$$
      and for every $-1<\beta\leq 0$ we have $$\varrho_{\alpha,\beta}(0)=\sqrt{\frac{\alpha+2+\beta}{1+\beta}}.$$
Moreover, using the previous proof we conclude the following Cauchy-Schwarz inequality:
\begin{cor}
    For every $\alpha\in\mathbb N$ and $\beta\in]-1,0]$ there exists $x_0\geq 1$ such that for every $x\in]-\infty,x_0[$ we have
    $$\left(\sum_{j=0}^{\alpha+1}  j\frac{(-x)^j}{j+\beta}{\alpha+1\choose{j}}\right)^2\leq \left(\sum_{j=0}^{\alpha+1}  j^2\frac{(-x)^j}{j+\beta}{\alpha+1\choose{j}}\right)\left(\sum_{j=0}^{\alpha+1} \frac{(-x)^j}{j+\beta}{\alpha+1\choose{j}} \right).
    $$

\end{cor}
\begin{proof}
    If $x\leq 0$ then the kernel $\frac{(-x)^j}{j-r}{\alpha+1\choose{j}}\geq0$
     for every $0\leq j\leq \alpha+1$ and the inequality is obvious. For $0<x<1$, thanks Inequality $\mathscr G_{Q_{\alpha,\beta}}(x)\geq0$ we obtain the desired Cauchy-Schwarz Inequality. To conclude the proof one can consider $x_0$ the supremum of all $x$ such that the Cauchy-Schwarz Inequality holds for $x$.
\end{proof}

For every piecewise continuously differentiable curve $\gamma:[0,1]\longrightarrow \mathbb D$, we define the length of $\gamma$ by
$$\ell_{\alpha,\beta}(\gamma):=\int_0^1\varrho_{\alpha,\beta}(\gamma(s))|\gamma'(s)|ds.
$$
One can define a distance on $\mathbb D$ by $\mathbf{d}_{\alpha,\beta}(p,q):=\inf\ell_{\alpha,\beta}(\gamma)$ where the infimum is taken over all piecewise continuously differentiable curves $\gamma$ on $\mathbb D$ with start point $p$ and end point $q$.
\begin{cor}
    For $-1<\alpha<+\infty$ and $-1<\beta\leq 0$, the space $(\mathbb D,\mathbf{d}_{\alpha,\beta})$ is a complete metric space.
\end{cor}
\begin{proof}
    This is a simple consequence of the fact that this metric is equivalent with the Poincar\'e metric on the disc $\mathbb D$. Indeed for every $z\in\mathbb D$ one has
    \begin{equation}\label{Pe}
    \frac{\sqrt{\alpha+2}}{1-|z|^2}\leq \varrho_{\alpha,\beta}(z)\leq \sqrt{\frac{\alpha+2+\beta}{1+\beta}}\frac1{1-|z|^2}
    \end{equation}

    and the two constants in these inequalities are sharp.
\end{proof}
\subsection{Toeplitz and Berezin operators via Carleson measures}
Toeplitz operators can also be defined for finite measures. Given a finite complex Borel measure $\nu$ on $\mathbb{D}$, we introduce $T^{\alpha,\beta}_\nu=T_\nu$ on $\mathcal{A}_{\alpha,\beta}^p$, for $-1<\alpha,\beta<+\infty$ as follows:
$$T_\nu f(z)=\int_{\mathbb{D}} f(w)\ \overline{\mathbb{K}_{\alpha,\beta}(w,z)}\ d\nu(w),\ \quad f\in  \mathcal{A}_{\alpha,\beta}^p.$$
Now, we characterize the boundedness and compactness of $T_\nu$ by the Carleson measure. Let $\nu$ be a finite positive Borel measure on $\mathbb{D}$ and $p\geq1$. We say that $\nu$ is a Carleson measure on the $\beta$-modified Bergman space $\mathcal{A}_{\alpha,\beta}^p$ if there exists a finite constant $C>0$ such that
$$\int_{\mathbb{D}} |f(z)|^p d\nu(z)\leq C \int_{\mathbb{D}} |f(z)|^p d\mu_{\alpha,\beta}(z),$$
for all $f\in \mathcal{A}_{\alpha,\beta}^p$.
We say that $\nu$ is vanishing Carleson on $\mathcal{A}_{\alpha,\beta}^p$ with $p>1$ if the inclusion mapping
$$i_p: \mathcal{A}_{\alpha,\beta}^p\rightarrow \textbf{L}^p(\mathbb{D},d\nu)$$
is compact and we say that it is vanishing Carleson on $\mathcal{A}_{\alpha,\beta}^1$ if
$$i_1: \mathcal{A}_{\alpha,\beta}^1\rightarrow \textbf{L}^1(\mathbb{D},d\nu)$$
is $*-$ compact means that $\left\|i_1(f_n)\right\|_{\textbf{L}^1(d\nu)}$ converges to $0$ for every sequence $f_n$ which converges to $0$ in the weak-star topology of $\mathcal{A}_{\alpha,\beta}^1$.\\
Now, we let $-1<\alpha<+\infty$ and $-1<\beta \leq 0$ and we define for  any $r>0$ and $z\in \mathbb{D}$ the modified Bergman disc as follow
$$\mathbf{D}_{\alpha,\beta}(z,r)=\{w\in \mathbb{D}^*; \mathbf{d}_{\alpha,\beta}(w,z)<r\},$$
where $\mathbf{d}_{\alpha,\beta}$ be the modified Bergman-Poincar\'{e} metric defined in the previous section. In the sequel of this part, we take $-1<\alpha<+\infty$ and $-1<\beta\leq 0$.

\begin{lem}\label{measure}
For any $r>0$ and $w\in \mathbb{D}^*$,  the following statements hold:

\begin{description}
	\item[$(1)$] Set $\tau_1= \tanh(\frac{r }{\sqrt{\alpha+2}})$ and $\tau_2= \tanh(\frac{r \sqrt{1+\beta}}{\sqrt{\alpha+2+\beta}})$, then
	$$ \ds\frac{(1-|z|^2)^2 \tau_2^2}{(1-|z|^2 \tau_2^2)^2}\leq |\mathbf{D}_{\alpha,\beta}(z,r)|\leq \frac{(1-|z|^2)^2 \tau_1^2} {(1-|z|^2 \tau_1^2)^2}.$$	
  \item [$(2)$] For every $\beta\in\mathcal J_\alpha $, we have
	$$\inf\left\{\left|k_{\alpha,\beta,2}^z(w)\right|^2;\ w\in \mathbf{D}_{\alpha,\beta}(z,r)\right\}\gtrsim\left[\mu_{\alpha,\beta}\big(\mathbf{D}_{\alpha,\beta}(z,r)\big)\right]^{-1}.$$
\end{description}

\end{lem}
\begin{proof}
The first statement can be deduced from Inequality \eqref{Pe} and Lemma $4.3.3$ in \cite{r9}. Now, we show the second statement in two steps. To start the first step, we claim that Inequality \eqref{Pe} gives
\begin{equation}\label{poincare0}
\mu_{\alpha,\beta}\big(\mathbf{D}(z,c_2r)\big)\leq \mu_{\alpha,\beta}\big(\mathbf{D}_{\alpha,\beta}(z,r)\big) \leq\mu_{\alpha,\beta}\big(\mathbf{D}(z,c_1r)\big),
\end{equation}
where $$c_1=\frac{1}{\sqrt{\alpha+2}},\ c_2=\sqrt{\frac{1+\beta}{\alpha+2+\beta}}$$ and $\mathbf{D}(z,s)$  is the classical (Poincar\'e) hyperbolic disc, that is $\mathbf{D}(z,s)=\mathbf{D}_{0,0}(z,s)$ corresponding to $\alpha=\beta=0$ in our statement.\\

Moreover, we have $\mathbf{D}(z,c_jr)=\mathbb{D}(C_j,R_j)$,  where
$$C_j=\frac{1-\tau^2_j}{1-\tau^2_j|z|^2}z,\qquad R_j=\frac{1-|z|^2}{1-\tau^2_j|z|^2}\tau^2_j,\quad j=1,2.$$
It follows that
\begin{align*}
\mu_{\alpha,\beta}\big(\mathbf{D}(z,c_2r)\big)&=\frac{1}{\mathscr{B}(\alpha+1,\beta+1)}\int_{\mathbb{D}(C_2,R_2) } |w|^{2\beta} (1-|w|^2)^\alpha dA(w)\\
&\gtrsim\int_{\mathbb{D}(C_2,R_2) }  (1-|w|^2)^\alpha dA(z)\\
& \gtrsim\mu_{\alpha,0}\big(\mathbb{D}(C_2,R_2)\big).
\end{align*}

Thanks to \cite[p. 121]{r9}, $\mu_{\alpha,0}(\mathbb{D}(C_2,R_2))$ is comparable to
$$|\mathbb{D}(C_2,R_2)|^{1+\frac{\alpha}{2}}=\ds\frac{(1-|z|^2)^{\alpha+2} \tau_2^{\alpha+2}}{(1-|z|^2 \tau_2^2)^{\alpha+2}}.$$ Thus we conclude that  $$\mu_{\alpha,\beta}(\mathbf{D}_{\alpha,\beta}(z,r))\gtrsim\frac{(1-|z|^2)^{\alpha+2} \tau_2^{\alpha+2}}{(1-|z|^2 \tau_2^2)^{\alpha+2}}.$$  This achieves the first step.

We claim that for every $\beta\in\mathcal J_\alpha$, the function $|Q_{\alpha,\beta}|$ has no zero in $\mathbf{D}_{\alpha,\beta}(z,r)$  which implies in particular that the infimum of this function will be finite and different to $0$. It follows that
\begin{align*}
    \left|k_{\alpha,\beta,2}^z(w)\right|^2&=\frac{|Q_{\alpha,\beta}(z\overline{w})|^2}{Q_{\alpha,\beta}(|z|^2)}\frac{(1-|z|^2)^{\alpha+2}}{|1-z\overline{w}| ^{2\alpha+4}}\\
    &\gtrsim \frac{(1-|z|^2)^{\alpha+2}}{|1-z\overline{w}| ^{2\alpha+4}}=\left|k_{\alpha,0,2}^z(w)\right|^2
\end{align*}
Again, applying Lemma 4.3.3 in \cite{r9}, we obtain
\begin{align*}
    \inf_{w\in\mathbf{D}_{\alpha,\beta}(z,r)}\left|k_{\alpha,\beta,2}^z(w)\right|^2&\geq \inf_{w\in\mathbf{D}(z,c_2r)}\left|k_{\alpha,\beta,2}^z(w)\right|^2 \\
    &\gtrsim \inf_{w\in\mathbf{D}(z,c_2r)}\left|k_{\alpha,0,2}^z(w)\right|^2\\
    &\gtrsim \frac{(1-|z| \tau_2)^{2\alpha+4}}{(1-|z|^2)^{\alpha+2}}\\
    &\gtrsim \frac{1}{\mu_{\alpha,\beta}(\mathbf{D}_{\alpha,\beta}(z,r))}
\end{align*}
and this finishes the proof.
\end{proof}
In the sequel, we assume that $\beta\in \mathcal{J}_{\alpha}$ and we give below a result concerning  Carleson measure that will be useful for the characterization of Toeplitz operators.
\begin{thm}\label{Carleson}
Let $\nu$ be a finite positive measure on $\mathbb{D}$, $p\geq1$ and $r>0$. If $\nu$ is a Carleson measure on $\mathcal{A}_{\alpha,\beta}^p$ then
$$\sup \bigg\{\frac{\nu(\mathbf{D}_{\alpha,\beta}(z,r))}{\mu_{\alpha,\beta}\big(\mathbf{D}_{\alpha,\beta}(z,r)\big)};\ z\in \mathbb{D}\bigg\}<+\infty.$$
\end{thm}
\begin{proof}
For the proof, we proceed as in \cite{r9}. To this purpose, we assume that $\nu$ is a Carleson measure on $\mathcal{A}_{\alpha,\beta}^p$. Then there exists a constant $\eta>0$ such that, for all $f\in \mathcal{A}_{\alpha,\beta}^p$, we have
$$\int_{\mathbb{D}} |f(w)|^p d\nu(w)\leq \eta \int_{\mathbb{D}} |f(w)|^p d\mu_{\alpha,\beta}(w).$$
Let $r>0$ and $z\in \mathbb{D}^*$. If we take $f(w)=(k_{\alpha,\beta,2}^z(w))^{2/p}$ then we obtain
\begin{align*}
\int_{\mathbf{D}_{\alpha,\beta}(z,r)} |k_{\alpha,\beta,2}^z(w)|^2\  d\nu(w)&\leq \int_{\mathbb{D}} |k_{\alpha,\beta,2}^z(w)|^2\  d\nu(w)\\
&\leq \eta \int_{\mathbb{D}} |k_{\alpha,\beta,2}^z(w)|^2\ d\mu_{\alpha,\beta}(w)\\
&\leq \eta,
\end{align*}
In virtue of Lemma \ref{measure}, we have
\begin{align*}
    \frac{\nu(\mathbf{D}_{\alpha,\beta}(z,r))}{\mu_{\alpha,\beta}\big(\mathbf{D}_{\alpha,\beta}(z,r)\big)}&\lesssim \nu(\mathbf{D}_{\alpha,\beta}(z,r)) \inf_{w\in \mathbf{D}_{\alpha,\beta}(z,r)}|k_{\alpha,\beta,2}^z(w)|^2\\
    &\lesssim \int_{\mathbf{D}_{\alpha,\beta}(z,r)} |k_{\alpha,\beta,2}^z(w)|^2\  d\nu(w)\leq \eta.
\end{align*}
Hence
$$\sup_{z\in \mathbb{D}} \frac{\nu(\mathbf{D}_{\alpha,\beta}(z,r))}{\mu_{\alpha,\beta}\big(\mathbf{D}_{\alpha,\beta}(z,r)\big)}<+\infty.$$
\end{proof}

The next theorem analyze the relation between bounded Toeplitz operator $T_\nu$ on $\mathcal{A}_{\alpha,\beta}^p$ and its Berezin symbol which defined by
$$\tilde{\nu}(z)=\left\langle T_\nu k_{\alpha,\beta,p}^z,k_{\alpha,\beta,p}^z\right\rangle_{\alpha,\beta}=\int_{\mathbb{D}} |k_{\alpha,\beta,p}^z(w)|^2\ d\nu(w).$$

\begin{thm}\label{Carleson0}
Let $\nu$ be a finite positive Borel measure on $\mathbb{D}$. If $\nu$ is a Carleson measure on $\mathcal{A}_{\alpha,\beta}^2$ then $T_\nu$ is bounded on $\mathcal{A}_{\alpha,\beta}^2$ which in turn implies that $\tilde{\nu}$ is a bounded function on $\mathbb{D}$.
\end{thm}
\begin{proof}
By definition of Carleson measure, there exists a constant $C>0$ such that in particular for every $f\in \mathcal{A}_{\alpha,\beta}^1$ we have
$$\int_{\mathbb{D}} |f(z)| d\nu(z)\leq C \int_{\mathbb{D}} |f(z)| d\mu_{\alpha,\beta}(z).$$
If we take $f$ and $g$ two functions in $\mathcal{A}_{\alpha,\beta}^2$, then
\begin{align*}
|\left\langle T_\nu f, g \right\rangle_{\alpha,\beta}|&\leq \int_{\mathbb{D}} |f(z) g(z)|\ d\nu(z)\\
&\leq C \int_{\mathbb{D}} |f(z) g(z)| d\mu_{\alpha,\beta}(z)\\
&\leq C \left\| f\right\|_{\alpha,\beta,2} \left\| g\right\|_{\alpha,\beta,2}.
\end{align*}
Therefore, $T_\nu$ is bounded on $\mathcal{A}_{\alpha,\beta}^2$. Now, suppose that the later statement holds. Since $k_{\alpha,\beta,2}^z$ is a unit vector in $\mathcal{A}_{\alpha,\beta}^2$, a direct application of Cauchy Schwartz inequality gives the boundedness of  $\tilde{\nu}$ on $\mathbb{D}$.
\end{proof}

For a bounded operator $S$ on the Bergman space $\mathcal{A}_{\alpha,\beta}^p$, we can define its Berezin transform as
$$\widetilde{S}(z)=\left\langle Sk_{\alpha,\beta,p}^z,k_{\alpha,\beta,q}^z\right\rangle_{\alpha,\beta}.$$
Note that the Berezin transform $\widetilde{T}_\varphi$ of Toeplitz operator $T_\varphi$, we will write $\widetilde{T}_\varphi=\widetilde{\varphi}$, is defined to be the Berezin transform of $\varphi$.
Fix $\varphi\in \textbf{L}^1(\mathbb{D}, d\mu_{\alpha,\beta})$ and $r>0$, we define the averaging function $\hat{\varphi}_r$ of $\varphi$ in the Bergman metric on $\mathbb{D}$ by
$$\hat{\varphi}_r(z)=\frac{1}{\mu_{\alpha,\beta}\big(\mathbf{D}_{\alpha,\beta}(z,r)\big)}\int_{\mathbf{D}_{\alpha,\beta}(z,r)} \varphi(w) d\mu_{\alpha,\beta}(w).$$

\begin{cor}
Let $\varphi$ be a positive function in $\textbf{L}^1(\mathbb{D}, d\mu_{\alpha,\beta})$. If $\hat{\varphi}_r$ is a bounded function on $\mathbb{D}$ then $T_\varphi$ is bounded on $\mathcal{A}_{\alpha,\beta}^2$ which in turn implies that $\tilde{\varphi}$ is a bounded function on $\mathbb{D}$.
\end{cor}

\begin{proof}
Is is an immediate consequence of Theorem \ref{Carleson} and Theorem \ref{Carleson0} by considering that $d\nu(z)=\varphi(z) d\mu_{\alpha,\beta}(z)$.
\end{proof}

\begin{thm}\label{Carleson1}
Let $\nu$ be a positive Borel measure on $\mathbb{D}$, $r>0$ and $p\geq 1$. If $\nu$ is a vanishing Carleson measure on $\mathcal{A}_{\alpha,\beta}^p$ then
$$\lim_{|z|\rightarrow 1^-} \frac{\nu(\mathbf{D}_{\alpha,\beta}(z,r))}{\mu_{\alpha,\beta}\big(\mathbf{D}_{\alpha,\beta}(z,r)\big)}=0.$$
\end{thm}
\begin{proof}
The proof is obvious. Indeed, we suppose that $\nu$ is a vanishing Carleson measure on $\mathcal{A}_{\alpha,\beta}^p$, then the inclusion mapping $i_p$ defined above is compact. On account of Proposition \ref{prop4.2}, one has
$$ \int_{\mathbf{D}_{\alpha,\beta}(z,r)} |k_{\alpha,\beta,2}^z(w)|^2\  d\nu(w)\leq \int_{\mathbb{D}} |k_{\alpha,\beta,2}^z(w)|^2\  d\nu(w)\rightarrow 0,$$
as $|z|$ tends to $1^-$. By Lemma \ref{measure}, we have
$$\lim_{|z|\rightarrow 1^-} \frac{\nu(\mathbf{D}_{\alpha,\beta}(z,r))}{\mu_{\alpha,\beta}\big(\mathbf{D}_{\alpha,\beta}(z,r)\big)}=0.$$

\end{proof}

\begin{thm}\label{Carleson2}
Let $\nu$ be a positive Borel measure on $\mathbb{D}$ and $p\geq 1$. If $\nu$ is a vanishing Carleson measure on $\mathcal{A}_{\alpha,\beta}^p$ then $T_\nu$ is compact on $\mathcal{A}_{\alpha,\beta}^2$ which implies that $\tilde{\nu} (z)$ tends to $0$ as $|z|$ tends to $1^-$.
	
\end{thm}

\begin{proof}
It's straightforward to see the second implication. As for the first one, we assume that
 $\nu$ is a vanishing Carleson measure on $\mathcal{A}_{\alpha,\beta}^p$ and we prove that $T_\nu$ is compact.
For every $f\in \mathcal{A}_{\alpha,\beta}^2$, we have
\begin{align*}
\left\|T_\nu f\right\|_{\alpha,\beta,2}&=\sup_{\substack{\left\|g\right\|_{\alpha,\beta,2}=1 \\ g\in \mathcal{A}_{\alpha,\beta}^2} } |\left\langle T_\nu f,g\right\rangle_{\alpha,\beta}|\\
&=\sup_{\substack{\left\|g\right\|_{\alpha,\beta,2}=1 \\ g\in \mathcal{A}_{\alpha,\beta}^2} } \left|\int_{\mathbb{D}} f(z)\overline{g(z)} d\nu(z)  \right|\\
&\leq \left\|f\right\|_{\textbf{L}^2(d\nu)} \sup_{\substack{\left\|g\right\|_{\alpha,\beta,2}=1 \\ g\in \mathcal{A}_{\alpha,\beta}^2} } \left\|g\right\|_{\textbf{L}^2(d\nu)}\\
&\leq C \left\|f\right\|_{\textbf{L}^2(d\nu)},
\end{align*}
where $C>0$. The last inequality is due to the fact that $\nu$ is a Carleson measure. Now, we take a sequence $(f_n)_{n\geq 0}$ such that it convergence weakly to $0$. Since $\nu$ is vanishing Carleson measure, the mapping implication $i_2: \mathcal{A}_{\alpha,\beta}^2\rightarrow\textbf{L}^2(d\nu)$ is compact which implies that $\left\|f\right\|_{\textbf{L}^2(d\nu)}\rightarrow 0$. Thereby, we obtain that $T_\nu$ is compact.

\end{proof}

\begin{cor}
Let $\varphi$ be a positive function in $\textbf{L}^1(\mathbb{D}, d\mu_{\alpha,\beta})$. If $\hat{\varphi}_r(z)$ tends to $0$ as $|z|$ tends to $1^-$ then $T_\varphi$ is compact on $\mathcal{A}_{\alpha,\beta}^2$ which in turn implies that $\tilde{\varphi}(z)$ tends to $0$ as $|z|$ tends to $1^-$.
\end{cor}
\begin{proof}
It is an immediate consequence of Theorem \ref{Carleson1} and \ref{Carleson2}.
\end{proof}
It should be mentioned at the end of this section that the Carleson measure approach to characterize the boundedness and compactness of a Toeplitz operator on Bergman spaces represent an interesting and special subject in its own right. There is a detailed study on Toeplitz operators and the Carleson measure on classical Bergman spaces, a nice overview of its properties and applications is given in a famous book of K.Zhu \cite{r9}. For this is why we wanted to give in this paper an overview on this theme on our space but in fact it is a brief overview because always of problem at the level of calculation. For example, the converse implication  of what may be called the key theorem to prove the other results, Theorem \ref{Carleson}, does not remain valid on the $\beta$-modified Bergman space since we could not pass the inequality found in Proposition 4.3.8 in \cite{r9} on $\mathcal{A}_{\alpha,\beta}^2$. Most of the proofs \cite{r9} are based on \cite[Lemma 4.3.3]{r9} that we tried to prove by modifying the measure and the Bergman disc but unfortunately we managed to show just one inequality (see Lemma \ref{measure}) the other will then be considered as an open question.

\end{document}